\documentclass[a4paper,11pt]{article}
\usepackage[
a4paper,
top=1in,
bottom=1in,
left=1in,
right=1in]{geometry}
\usepackage[scaled=.92]{helvet}

\usepackage{hyperref}

\usepackage{graphicx} 
\usepackage{ytableau}
\usepackage{mathrsfs}
\usepackage{tikz-cd}
\usepackage{amsmath}
\usepackage{quiver}
\usepackage[dvipsnames]{xcolor}
\usepackage{amsthm}
\usepackage{amssymb}
\usepackage{comment}
\usepackage{float}
\renewcommand\labelenumi{\upshape(\roman{enumi})}
\renewcommand\theenumi\labelenumi

\newtheorem{theorem}{Theorem}[section]
\newtheorem{lemma}[theorem]{Lemma}

\newtheorem{proposition}[theorem]{Proposition}
\newtheorem{conjecture}[theorem]{Conjecture}

\theoremstyle{definition} 
\newtheorem{example}[theorem]{Example}

\newtheorem{remark}[theorem]{Remark}
\newtheorem{definition}[theorem]{Definition}
\newtheorem{claim}[theorem]{Claim}

\title{Extending the Bipartite Parking Space}
\date{}
\author{Dora Woodruff}

\begin{document}

\maketitle

\begin{abstract}
    We prove an analogue of a theorem of Berget and Rhoades about extending the parking space $\mathrm{Park}_n$ to an $S_{n+1}$-module $\mathrm{Slim}_n$. Specifically, we show that the \textit{bipartite parking space} $\mathrm{Park}_{K_{n,m}}$, which naturally comes with an $S_{n-1} \times S_m$ action, extends to an $S_n \times S_m$-representation $\mathrm{Slim}_{n,m}$. We then formulate a conjecture generalizing this statement to \textit{any} simple graph.
\end{abstract}

\section{Introduction}

Let $G$ be a graph with a distinguished vertex $v$. Its automorphism group is $\mathrm{aut}(G)$, and $\mathrm{aut}_v(G)$ is the subgroup of automorphisms fixing $v$. To any such $G$, we can associate two important representations. 

First, there is the \textit{parking space} $\mathrm{Park}_G$, which is a representation of $\mathrm{aut}_v(G)$. This $\mathrm{Park}_G$ is spanned by Postnikov and Shapiro's \textit{graphical parking functions} \cite{postnikovshapiro}, and its dimension is the number of spanning trees of $G$. 

Second, there is an $\mathrm{aut}(G)$-representation which we denote by $\mathrm{Slim}_G$ (Definition \ref{def:slim}). This space arises from the notion of \textit{slim subgraphs}. Postnikov and Shapiro \cite{postnikovshapiro} have studied $\mathrm{Slim}_G$ as a graded vector space, but not as a representation. 

When $G$ is the complete graph $K_n$, $\mathrm{Park}_G$ is a central object in algebraic combinatorics: the (classical) parking space $\mathrm{Park}_{n-1}$, where $S_{n-1}$ acts by permuting entries of length $n-1$ parking functions \cite{parkingfirst}. The space $\mathrm{Park}_{n-1}$ is connected to diagonal coinvariant rings, Catalan combinatorics, rational Cherednik algebras, and much more (see \cite{cherednik}, \cite{hyperplanes}, \cite{haiman}). Motivated by these rich connections, many variants of classical parking functions have been studied. For example, Armstrong, Loehr, and Warrington \cite{rationalpf} defined and studied \textit{rational parking functions}. 

On the other hand, even when $G = K_n$, $\mathrm{Slim}_{K_n} := \mathrm{Slim}_n$ is less well understood. For example, it is still an open problem to give a nice expression for the graded Frobenius series of $\mathrm{Slim}_{n}$. A central difficulty is that $\mathrm{Slim}_{n}$ does not have an obvious $S_n$-invariant basis (while $\mathrm{Park}_n$ does). Nonetheless, Berget and Rhoades \cite{bergetrhoades} were able to prove that $\mathrm{Slim}_{n}$ extends the parking space $\mathrm{Park}_{n-1}$:

\begin{theorem}[Berget-Rhoades] \label{thm:berget-rhoades}
    Embed $S_{n-1}$ into $S_n$ by letting $S_{n-1}$ act on the first $n-1$ letters. Then,
    \[\mathrm{Res}^{S_n}_{S_{n-1}}(\mathrm{Slim}_{n}) \simeq_{S_{n-1}} \mathrm{Park}_{n-1}\]
\end{theorem}

More recently, Konvalinka and Tewari \cite{tewari} conjectured a more explicit description of $\mathrm{Slim}_{n}$, motivated by understanding the Berget-Rhoades extension. In particular, the representation they construct does have an explicit $S_n$-invariant basis, and its character is thus more readily computable. 

The goal of this article is to prove an analogue of the Berget-Rhoades extension for complete bipartite graphs. (In fact, we conjecture that an analogue of the Berget-Rhoades extension should hold for \textit{all} graphs, see Conjecture \ref{conj}). As we will see, the combinatorics of $\mathrm{Park}_{K_{n,m}} := \mathrm{Park}_{n,m}$ is quite rich. The classical $\mathrm{Park}_{n}$ is governed by Catalan combinatorics; on the other hand, the combinatorics of $\mathrm{Park}_{n,m}$ is governed by \textit{Narayana} combinatorics. Using Narayana-inspired tools, we prove our main result:

\begin{theorem}\label{thm:main}
    Embed $S_{n-1}$ into $S_n$ by letting $S_{n-1}$ act on the first $n$ letters. Then:
    \[\mathrm{Res}^{S_n \times S_m}_{S_{n-1} \times S_m} (\mathrm{Slim}_{K_{n,m}}) \simeq_{S_{n-1} \times S_m} \mathrm{Park}_{K_{n,m}}\]
\end{theorem}

The rest of the paper is structured as follows. Section \ref{sec:background} fixes notation and discusses some background on $K_{n,m}$-parking functions. In particular, we explain a useful lattice path model of Snider and Yan \cite{U} (originally described in the language of $(p,q)$-parking functions, but we translate it to the equivalent setting of $K_{n,m}$-parking functions). This lattice path model shows that the number of orbits in $\mathrm{Park}_{K_{n,m}}$ is the Narayana number $N_{n, n+m}$, a fact previously discovered by Cori and Poulalhon \cite{narayana}. Section \ref{sec:proof} gives the proof of Theorem \ref{thm:main}. The technical heart of this proof is a new partial order on pairs of nonintersecting lattice paths, described in Subsection \ref{subsec:partial}. We end with some remarks on extending Theorem \ref{thm:main} to more general graphs. 

\section{Background}\label{sec:background}

\subsection{Graphical parking functions and Postnikov-Shapiro algebras}

In this subsection, we review Postnikov and Shapiro's definition of \textit{graphical parking functions}. We recall how they viewed $\mathrm{Park}_G$ as a certain quotient of $\mathbb{C}[x_1, \dots x_{n}]$. Then, we review their definition of $\mathrm{Slim}_G$. 

Throughout, let $G$ be a (simple, undirected) graph with vertex set labeled by $V(G) = [n+1]$ and distinguished vertex $n+1$. If $S \subseteq V(G)$ and $s \in S$, we set $d_S(s)$ as the number of edges with endpoints $(s, s')$, where $s' \notin S$ (the number of edges going from $s$ to outside of $S$).

\begin{definition}[Graphical parking functions]\label{def:graphical_parking}
    A \textit{$G$-parking function} is a tuple $(a_1, a_2 \dots a_n)$ such that for every $S \subseteq (V(G)\setminus \{v\})$, there exists $s \in S$ such that 
    \[a_s < d_S(s)\]
\end{definition}

When $G = K_{n+1}$, the $G$-parking functions are exactly the classical length-$n$ parking functions. From Definition \ref{def:graphical_parking}, it is clear that any $G$-automorphism fixing the distinguished vertex $v$ preserves $G$-parking functions.  

Postnikov and Shapiro also gave an algebraic reformulation of Definition \ref{def:graphical_parking}:

\begin{definition}
    Let $\mathcal{I}_G = \langle m_S \rangle$ be the monomial ideal of $\mathbb{C}[x_1 \dots x_n]$ generated by $m_S$, where 
    \[m_S = \prod_{s \in S} x_s^{d_S(s)}\]
    as $S$ ranges over all subsets of $V(G) \setminus \{v\}$. Then, define
    \[\mathrm{Park}_G = \mathbb{C}[x_1, \dots, x_n]/\mathcal{I}_G\]
\end{definition}

For example, when $G = K_{n+1}$, $\mathrm{Park}_G$ is the quotient of $\mathbb{C}[x_1, \dots, x_n]$ by all substaircase monomials. Monomials of the form $x^{\vec{a}}$, where $\vec{a}$ is a $G$-parking function, form a linear basis for $\mathrm{Park}_{G}$. Therefore, $\mathrm{Park}_G$ comes with an action of $\mathrm{aut}_v(G)$ that permutes this monomial basis. 

Having defined $\mathrm{Park}_G$, we now define $\mathrm{Slim}_G$. A subgraph $H \subseteq G$ is called \textit{slim in $G$} if its edge complement in $G$ is connected. To an edge $e$ of $G$, we associate the polynomial 
\[p(e) = x_i - x_j\]
where $i < j$ are the labels of the endpoint of $e$. Then, we can associate a polynomial $p(H)$ to a subgraph $H$:
\[p(H) := \prod_{e \in H} p(e)\]

\begin{definition}\label{def:slim}
    The algebra $\mathrm{Slim}_G$ is the $\mathbb{C}$-linear subspace of $\mathbb{C}[x_1, \dots, x_n]$ spanned by 
    \[\{p(H)| H \text{ slim in } G\}\]
\end{definition}

Note that $p(H)$ is homogeneous, with degree the number of edges in $H$; therefore, $\mathrm{Slim}_G$ is naturally a \textit{graded} algebra. Note also that \textit{all} automorphisms of $G$ preserve the property of slimness; therefore, $\mathrm{Slim}_G$ is naturally an $\mathrm{aut}(G)$-representation. 

Postnikov-Shapiro studied $\mathrm{Park}_G$ and $\mathrm{Slim}_G$ as graded vector spaces. They proved: 

\begin{theorem}[Postnikov-Shapiro]
    The dimensions of both $\mathrm{Slim}_G$ and $\mathrm{Park}_G$ are equal to the number of spanning trees of $G$. 
\end{theorem}

To prove this theorem, Postnikov and Shapiro constructed, for any linear ordering of $V(G)$, an explicit matoid-theoretic basis for $\mathrm{Slim}_G$. However, this basis is not usually invariant under the action of $\mathrm{aut}(G)$, so in this paper we make no further use of it. 

Having defined $\mathrm{Slim}_G$ and $\mathrm{Park}_G$, we next specialize to the case of $G = K_{n,m}$. 

\subsection{Bipartite parking functions and nonintersecting lattice paths}

Snider and Yan \cite{U} (and previously, Cori and Poulalhon \cite{narayana}) studied the enumeration of $(p,q)$-parking functions. These $(p, q)$-parking functions are essentially (modulo convention swaps) the same thing as $K_{n,m}$-parking functions. In this section, we recall their results and explain why the orbits of $\mathrm{Park}_{K_{n,m}}$ are enumerated by the Narayana numbers. 

\begin{remark}
    There are many generalizations of parking functions, and many of these generalizations overlap substantially (as does their notation, unfortunately). Snider and Yan \cite{clarification1} helped clarify the connections between $U$-parking functions, vector parking functions, graphical parking functions, and $(p,q)$-parking functions. From now on, we ignore these differing conventions and families, and stick to our definition of $K_{n,m}$-parking functions. 

    We choose the terminology `bipartite parking functions' over, say, $(p,q)$-parking functions to make the connection to $\mathrm{Slim}_G$ clearer, as well as to phrase conjectural generalization to all graphs $G$. 
\end{remark}

We will always label the \textit{righthand} vertices of $K_{n,m}$ by $1, 2,\dots, m$, and the \textit{lefthand} vertices by $m+1, m+2 \dots m+n$. The distinguished root is always vertex $m+n$. 

\begin{definition}
    A $K_{n,m}$-parking function is \textit{nonincreasing} if
    \[a_1 \geq a_2 \dots \geq a_m\]
    and 
    \[a_{m+1} \geq a_{m+2} \dots \geq a_{n+m-1}\]
\end{definition}

Representation theoretically, orbits of $\mathrm{Park}_{n,m}$ are represented by nonincreasing $K_{n,m}$-parking functions. 

\begin{theorem}[Snider-Yan]
    Nonincreasing $K_{n,m}$-parking functions are in bijection with pairs of \textit{weakly nonintersecting lattice paths} in an $m \times (n-1)$ grid. In particular, the number of nonincreasing $K_{n,m}$-parking functions is $N(m, n+m-1)$, the number of Dyck paths of semilength $m+n-1$ with $m$ peaks. 
\end{theorem}

Here, by \textit{weakly nonintersecting} lattice paths, we mean a pair of paths $(P, Q)$ such that $Q$ never goes above $P$, and $P$ and $Q$ share no \textit{vertical} steps (although they may share horizontal steps!) 

Let us make this correspondence between nonincreasing $K_{n,m}$-parking functions and weakly nonintersecting pairs explicit. Let $\vec{a}$ be a $K_{n,m}$-parking function. Read its first $m$ coordinates corresponding to its $m$ righthand vertices; these coordinates will define the path $Q$. By applying Definition \ref{def:graphical_parking} to a singleton $S$, we must have $a_i < n$ for $i \leq m$. Construct the auxiliary sequence $(b_0, b_1 \dots b_{n-1})$, where $b_i = \#\{j | a_j = i\}$. Then, $Q$ is the lattice path 

\[E^{b_0}NE^{b_1} NE^{b_2} \dots E^{b_{n-1}}\]

The `upper path' $P$ is constructed similarly, but we swap the roles of horizontal and vertical steps. Explicitly: after reading the coordinates $(a_{m+1}, \dots, a_{n+m-1})$, we construct the analogous auxiliary sequence $(b_0', b_1' \dots b_{m-1}')$. Then, $P$ is the lattice path 

\[N^{b_0'}EN^{b_1'} E \dots N^{b_{m-1}'}\]

\begin{proposition}[Snider-Yan]
    Encoding length $m+n-1$ tuples as lattice path pairs $(P, Q)$ in the way described above, $\vec{a}$ is a valid $K_{n,m}$-parking function if and only if $(P, Q)$ is a weakly nonintersecting pair. 
\end{proposition}

From here on, if $\vec{a}$ is a tuple, we denote its corresponding lattice path pair by $(P_{\vec{a}}, Q_{\vec{a}})$. 

\begin{example}
    Let $\vec{a} = (2, 2, 1, 0, 3, 2, 0, 0)$. To check whether $\vec{a}$ is a $K_{5,4}$-parking function, we first draw the pair $(P_{\vec{a}}, Q_{\vec{a}})$, which lives in a $4 \times 4$ grid. 

    The auxiliary sequence $(b_0, b_1, b_2, b_3, b_4)$ for the righthand vertices is $(1, 1, 2, 0, 0)$. Therefore:
    \[Q_{\vec{a}} = ENENEN^2EN\]
    Meanwhile, the auxiliary sequence for the lefthand vertices is $(2,0,1,1)$. Therefore:
    \[P_{\vec{a}} = N^2E^2NENE\]
    Drawing the two paths (and shading the areas under and above the paths for clarity), we get: 

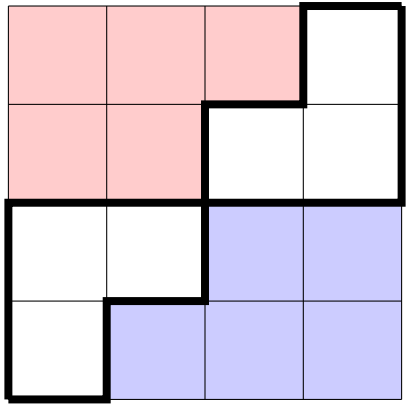
\begin{center}
\begin{tikzpicture}[scale=1.3]

\fill[blue!20]
  (0,0) -- (1,0) -- (1,1) -- (2,1) -- (2,2) --
  (4,2) -- (4,0) -- cycle;

\fill[red!20]
  (0,0) -- (0,2) -- (2,2) -- (2,3) -- (3,3) --
  (3,4) -- (4,4) -- (4,2) -- (4,4) --
  (0,4) -- cycle;

\foreach \x in {0,...,4}
    \draw (\x,0) -- (\x,4);

\foreach \y in {0,...,4}
    \draw (0,\y) -- (4,\y);

\draw[line width=3]
  (0,0) -- (1,0) -- (1,1) -- (2,1) -- (2,2) --
  (4,2) -- (4,4);

\draw[line width=3]
  (0,0) -- (0,2) -- (2,2) -- (2,3) -- (3,3) --
  (3,4) -- (4,4);

\end{tikzpicture}

\end{center}

    We can see that $(P_{\vec{a}}, Q_{\vec{a}})$ is indeed a \textit{weakly} nonintersecting pair (although they do touch at one point inside the grid). Therefore, $\vec{a}$ is a $K_{5,4}$-parking function. 
\end{example}

\section{Proof of Theorem \ref{thm:main}}\label{sec:proof}

As noted previously, it is difficult to directly compute the character of $\mathrm{Slim}_{n,m}$. Instead of comparing characters, we therefore construct an explicit isomorphism between $\mathrm{Park}_{n,m}$ and $\mathrm{Res}^{S_{n} \times S_{m}}_{S_{n-1} \times S_m} (\mathrm{Slim}_{n,m})$. As a quick overview: there is a natural $\mathrm{aut}_v(G)$-equivariant map $\phi_{n,m}: \mathrm{Slim}_{n,m} \to \mathrm{Park}_{n,m}$, so our goal is to show that $\phi_{n,m}$ is a bijection. This boils down to showing that a certain transition matrix is upper-triangular with $\pm 1$ on the diagonal. The most interesting step of the proof is constructing the right ordering for the rows and columns of this transition matrix, which is a certain partial order $\prec_{n,m}$ on the set of weakly nonintersecting lattice paths. We call this partial order the \textit{row-column ordering}.

\subsection{The homomorphism \texorpdfstring{$\phi_{n,m}$}{phi n,m}} \label{subsec:phi}

Recall that $\mathrm{Slim}_{n,m}$ is a subalgebra of the polynomial ring $\mathbb{C}[x_1 \dots x_{n+m}]$. Also recall that $\mathrm{Park}_{n,m}$ is isomorphic to a certain quotient of $\mathbb{C}[x_1 \dots x_{n+m-1}]$ by an ideal $\mathcal{I}_{K_{n,m}}$. 

\begin{definition}
    The map $\phi_{n,m}$ is the composition 
    \begin{center}
    \begin{tikzcd}
	{\mathrm{Slim}_{n,m}} && {\mathbb{C}[x_1 \dots x_{n+m}]} && {\mathbb{C}[x_1 \dots x_{n+m-1}]} & {\mathbb{C}[x_1 \dots x_{n+m-1}]/\mathcal{I}_{K_{n,m}}}
	\arrow[hook, from=1-1, to=1-3]
	\arrow[two heads, from=1-3, to=1-5]
	\arrow[two heads, from=1-5, to=1-6]
\end{tikzcd}

\end{center}
where the first map is inclusion, the second specializes $x_{n+m}$ to zero, and the third quotients out by $\mathcal{I}_{K_{n,m}}$. 

\end{definition}

It is clear that $\phi_{n,m}$ is $\mathrm{aut}_v(G)$-equivariant. Furthermore, Postnikov and Shapiro \cite{postnikovshapiro} showed that the dimension of both $\mathrm{Slim}_{n,m}$ and $\mathrm{Park}_{n,m}$ is equal to the number of spanning trees of $K_{n,m}$. Therefore, it suffices to show that $\phi_{n,m}$ is surjective. (Notice that everything so far applies to an arbitrary graph, not just $K_{n,m}$). 

Orbits of $\mathrm{Park}_{n,m}$ are labeled by \textit{nonincreasing} $K_{n,m}$-parking functions $\vec{a}$. Therefore, it suffices to show:

\begin{lemma}\label{lemma:surjective}
    For each \textit{nonincreasing} $K_{n,m}$-parking function $\vec{a}$, the monomial $x^{\vec{a}}$ is in the image of $\phi_{n,m}$.  
\end{lemma}

\subsection{Slim subgraphs from nonincreasing \texorpdfstring{$K_{n,m}$}{K n,m}-parking functions}

In this section we describe, for every nonincreasing $K_{n,m}$-parking function $\vec{a}$, an associated slim subgraph $H(\vec{a}) \subseteq K_{n,m}$ (and the corresponding polynomial $p(\vec{a}) \in \mathrm{Slim}_{n,m}$).

To define $H(\vec{a})$, we first describe a funny way of labeling grid squares in the $(n-1) \times m$ grid. We will need this grid square labeling throughout the remaining sections: 

\begin{definition}[Grid square labeling]\label{def:labeling}
    Let $(P, Q)$ be a weakly nonintersecting pair in an $m \times (n-1)$ grid (where path $P$ is above $Q$). Label the columns with $1, 2 \dots, m$ from right to left, and the rows with $1, 2, \dots, n-1$ from bottom to top. Then:
    
    \begin{enumerate}
        \item For each square \textit{below} $Q$ in column $i$ and row $j$, label the square with $(m-i+1, n+m-j+1)$. 
        \item For each square \textit{above} $P$ in column $i$ and row $j$, label the square with $(m+n-j+1, m-i+1)$.
    \end{enumerate}

    The resulting labeling is the \textit{grid square labeling of $(P_{\vec{a}}, Q_{\vec{a}})$.}
\end{definition}

Example \ref{example:labeling} shows an example of the grid square labeling described in Definition \ref{def:labeling}. 

    \begin{example}\label{example:labeling}
    Let $m=6, n = 8$, and $\vec{a} = (6,5,5,3,3,3,2,2,3,3,2,2,2,2,2,1)$. Then, $(P_{\vec{a}}, Q_{\vec{a}})$ is shown below, along with the grid square labeling described in Definition \ref{def:labeling}. 
    
    \begin{center}
    \begin{tikzpicture}[scale=1.2]
    \foreach \x in {0,...,6}
        \draw (\x,0) -- (\x,8);

    \foreach \y in {0,...,8}
        \draw (0,\y) -- (6,\y);

    \draw[line width = 3] (0,0) --++ ({90}:1) --++ ({0}:1)--++ ({90}:5)--++ ({0}:2) --++ ({90}:2) --++ ({0}:3);

    \draw[line width = 3]
    (0,0)--++ ({0}:2) --++ ({90}:2) --++ ({0}:1) --++ ({90}:3)--++ ({0}:2) --++ ({90}:2) --++ ({0}:1) --++ ({90}:1);

    \draw (0.5, 1.5) node{$(13,6)$};
    \draw (0.5, 2.5) node{$(12,6)$};
    \draw (0.5, 3.5) node{$(11,6)$};
    \draw (0.5, 4.5) node{$(10,6)$};
    \draw (0.5, 5.5) node{$(9,6)$};
    \draw (0.5, 6.5) node{$(8,6)$};
    \draw (0.5, 7.5) node{$(7,6)$};

    \draw (1.5, 7.5) node{$(7,5)$};
    \draw (1.5, 6.5) node{$(8,5)$};
    \draw (2.5, 7.5) node{$(7,4)$};
    \draw (2.5, 6.5) node{$(8,4)$};

    \draw (2.5, 0.5) node{$(4,15)$};
    \draw (2.5, 1.5) node{$(4,14)$};

    \draw (3.5, 0.5) node{$(3,15)$};
    \draw (3.5, 1.5) node{$(3,14)$};
    \draw (3.5, 2.5) node{$(3,13)$};
    \draw (3.5, 3.5) node{$(3,12)$};
    \draw (3.5, 4.5) node{$(3,11)$};

    \draw (4.5, 0.5) node{$(2,15)$};
    \draw (4.5, 1.5) node{$(2,14)$};
    \draw (4.5, 2.5) node{$(2,13)$};
    \draw (4.5, 3.5) node{$(2,12)$};
    \draw (4.5, 4.5) node{$(2,11)$};

    \draw (5.5, 0.5) node{$(1,15)$};
    \draw (5.5, 1.5) node{$(1,14)$};
    \draw (5.5, 2.5) node{$(1,13)$};
    \draw (5.5, 3.5) node{$(1,12)$};
    \draw (5.5, 4.5) node{$(1,11)$};
    \draw (5.5, 5.5) node{$(1,10)$};
    \draw (5.5, 6.5) node{$(1,9)$};

\end{tikzpicture}
    \end{center}

    \end{example}

\begin{definition}\label{def:H}
    Let $\vec{a}$ be a nonincreasing $K_{n,m}$-parking function. Its \textit{associated graph} $H(\vec{a})$ is the subgraph of $K_{n,m}$ where $(i,j)$ is included as an edge if and only if $(i,j)$ is a label for the grid square labeling of $(P_{\vec{a}}, Q_{\vec{a}})$. Its \textit{associated polynomial} $p(\vec{a}) \in \mathrm{Slim}_{n,m}$ is 
    \[p(H(\vec{a})) = \prod_{(i,j) \in e(H)} (x_i-x_j)\]
\end{definition}

\begin{example}
Let $m = 6, n = 7$, and  $\vec{a} = (7,5,2,2,1,0,4,4,3,3,3,1,0)$. In Figure \ref{fig:slim}, on the left is $(P_{\vec{a}}, Q_{\vec{a}})$ along with its grid square labeling, and on the right is the associated slim graph $H(\vec{a})$.

\end{example}

\begin{figure}
\centering
        \begin{tikzpicture}[scale=1.2]
    \foreach \x in {0,...,6}
        \draw (\x,0) -- (\x,7);

    \foreach \y in {0,...,7}
        \draw (0,\y) -- (6,\y);

    \draw[line width = 3] (0,0) --++ ({0}:1) --++ ({90}:1) --++ ({0}:1) --++ ({90}:1) --++ ({0}:2) --++ ({90}:3) --++ ({0}:1) --++ ({90}:2) --++ ({0}:1);

    \draw[line width = 3]
    (0,0) --++ ({90}:1) --++ ({0}:1)--++ ({90}:1) --++ ({0}:2) --++ ({90}:3) --++ ({0}:1) --++ ({90}:2) --++ ({0}:2);

    \draw (5.5, 0.5) node{$(1,14)$};
    \draw (5.5, 1.5) node{$(1,13)$};
    \draw (5.5, 2.5) node{$(1,12)$};
    \draw (5.5, 3.5) node{$(1,11)$};
    \draw (5.5, 4.5) node{$(1,10)$};
    \draw (5.5, 5.5) node{$(1,9)$};
    \draw (5.5, 6.5) node{$(1,8)$};

    \draw (4.5, 0.5) node{$(2,14)$};
    \draw (4.5, 1.5) node{$(2,13)$};
    \draw (4.5, 2.5) node{$(2,12)$};
    \draw (4.5, 3.5) node{$(2,11)$};
    \draw (4.5, 4.5) node{$(2,10)$};

    \draw (3.5, 0.5) node{$(3,14)$};
    \draw (3.5, 1.5) node{$(3,13)$};
    \draw (2.5, 0.5) node{$(4,14)$};
    \draw (2.5, 1.5) node{$(4,13)$};

    \draw (1.5, 0.5) node{$(5,14)$};

    \draw (0.5, 1.5) node{$(12,6)$};
    \draw (0.5, 2.5) node{$(11,6)$};
    \draw (0.5, 3.5) node{$(10,6)$};
    \draw (0.5, 4.5) node{$(9,6)$};
    \draw (0.5, 5.5) node{$(8,6)$};
    \draw (0.5, 6.5) node{$(7,6)$};

    \draw (1.5, 2.5) node{$(11,5)$};
    \draw (1.5, 3.5) node{$(10,5)$};
    \draw (1.5, 4.5) node{$(9,5)$};
    \draw (1.5, 5.5) node{$(8,5)$};
    \draw (1.5, 6.5) node{$(7,5)$};

    \draw (2.5, 2.5) node{$(11,4)$};
    \draw (2.5, 3.5) node{$(10,4)$};
    \draw (2.5, 4.5) node{$(9,4)$};
    \draw (2.5, 5.5) node{$(8,4)$};
    \draw (2.5, 6.5) node{$(7,4)$};

    \draw (3.5, 5.5) node{$(8,3)$};
    \draw (3.5, 6.5) node{$(7,3)$};

    \draw (8.5, 6.5) node{$7 \hspace{1em} \bullet$};
    \draw (8.5, 5.5) node{$8 \hspace{1em} \bullet$};
    \draw (8.5, 4.5) node{$9 \hspace{1em} \bullet$};
    \draw (8.5, 3.5) node{$10 \hspace{1em} \bullet$};
    \draw (8.5, 2.5) node{$11 \hspace{1em} \bullet$};
    \draw (8.5, 1.5) node{$12 \hspace{1em} \bullet$};
    \draw (8.5, 0.5) node{$13 \hspace{1em} \bullet$};
    \draw (8.5, -0.5) node{$14 \hspace{1em} \bullet$};

    \draw (10.5, 6.5) node{$\bullet \hspace{1em} 1$};
    \draw (10.5, 5.5) node{$\bullet \hspace{1em} 2$};
    \draw (10.5, 4.5) node{$\bullet \hspace{1em} 3$};
    \draw (10.5, 3.5) node{$\bullet \hspace{1em} 4$};
    \draw (10.5, 2.5) node{$\bullet \hspace{1em} 5$};
    \draw (10.5, 1.5) node{$\bullet \hspace{1em} 6$};

    \draw (10.25, 6.5) -> (8.75, 6.5);
    \draw (10.25, 5.5) -> (8.75, 6.5);
    \draw (10.25, 5.5) -> (8.75, 5.5);
    \draw (10.25, 5.5) -> (8.75, 4.5);
    \draw (10.25, 5.5) -> (8.75, 4.5);
    \draw (10.25, 4.5) -> (8.75, 4.5);
    \draw (10.25, 4.5) -> (8.75, 3.5);
    \draw (10.25, 4.5) -> (8.75, 2.5);
    \draw (10.25, 4.5) -> (8.75, 1.5);
    \draw (10.25, 3.5) -> (8.75, 1.5);
    \draw (10.25, 2.5) -> (8.75, 1.5);
    \draw (10.25, 2.5) -> (8.75, 0.5);
    \draw (10.25, 1.5) -> (8.75, 0.5);
    \draw (10.25, 1.5) -> (8.75, -0.5);

\end{tikzpicture}
    \caption{Lattice paths $(P_{\vec{a}}, Q_{\vec{a}})$ and the corresponding slim subgraph $H(\vec{a})$. In this case, $H(\vec{a})$ is a `zig-zag shaped' spanning tree.}\label{fig:slim}
\end{figure}
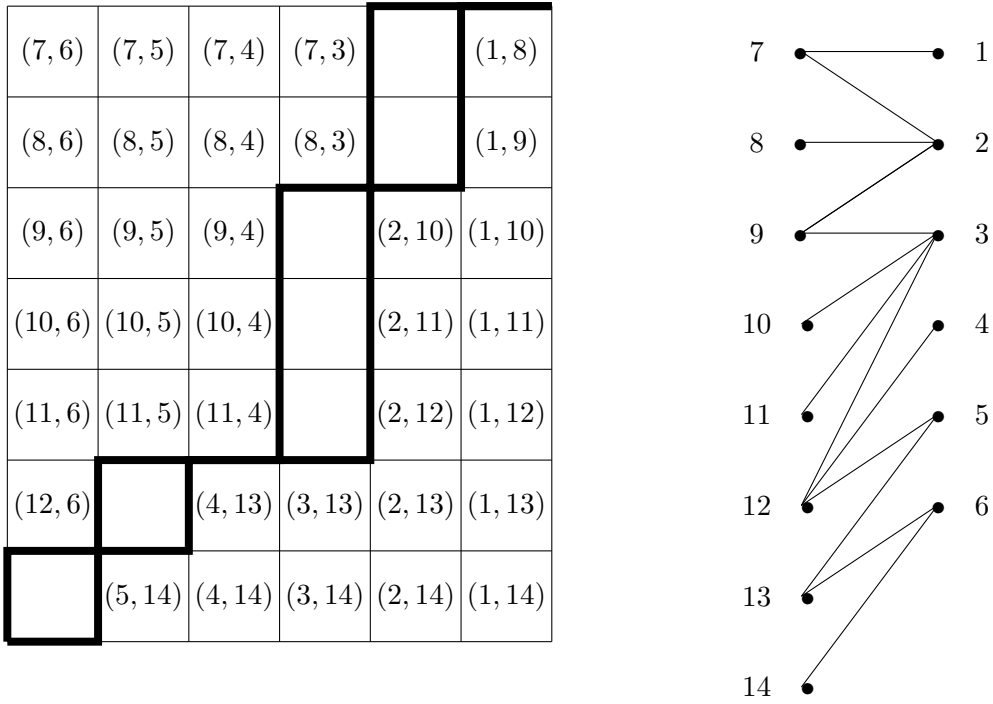

One can check that for all grid square labels described in Definition \ref{def:labeling}, exactly one of the two coordinates is $\leq m$. Therefore, $H(\vec{a})$ is a valid subgraph of $K_{n,m}$. It remains to verify that it is a \textit{slim} subgraph. 

\begin{lemma}
    For any nonincreasing $K_{n,m}$-parking function $\vec{a}$, $H(\vec{a})$ is a \textit{slim} subgraph of $K_{n,m}$. 
\end{lemma}

\begin{proof}
    Consider the lower path $Q_{\vec{a}}$. Suppose its horizontal runs have lengths $(b_0, b_1 \dots b_{n-1})$ (with $b_0 > 0$; otherwise, $(P_{\vec{a}}, Q_{\vec{a}})$ was not a weakly intersecting pair). Then, let $P_{\vec{a}}'$ be the unique lattice path with horizontal runs $(b_0-1, b_1, b_2 \dots b_{n-1}+1)$. (That is, $P_{\vec{a}}'$ `trails $Q_{\vec{a}}$ as closely as possible' without sharing any vertical steps with it). Let the slim graph associated to the pair $(P_{\vec{a}}', Q_{\vec{a}})$ be $H'(\vec{a})$. 

    It is clear that $H(\vec{a})$ is a subgraph of $H'(\vec{a})$ (that is, the set of squares above $P_{\vec{a}}$ is a subset of the set of squares above $P'_{\vec{a}}$). Therefore, to check that $H(\vec{a})$ is slim, it suffices to check that $H'(\vec{a})$ is slim. 

    But we can directly compute the edge-complement of $H'(\vec{a})$ in $K_{n,m}$: it is a `zig-zag-shaped' spanning tree of $K_{n,m}$ (see Figure \ref{fig:slim} for an example). Since this spanning tree is always connected, $H'(\vec{a})$ is slim. 
    
\end{proof}

We have now defined a specific element $p(\vec{a})$ of $\mathrm{Slim}_{n,m}$ which we hope `maps onto' $x^{\vec{a}}$ under $\phi_{n,m}$. It is not literally true that $\phi_{n,m}(p(\vec{a})) = x^{\vec{a}}$, but it is true up to lower order terms. The next subsection makes this notion precise by introducing the ordering $\prec_{n,m}$ on the set of $K_{n,m}$-parking functions. 

\subsection{The partial order \texorpdfstring{$\prec_{n,m}$}{prec n,m}} \label{subsec:partial}

To show that $\phi_{n,m}$ is surjective, we will consider the transition matrix between the monomial basis $\{x^{\vec{a}}\}$ of $\mathrm{Park}_{n,m}$ and polynomials the $\{\phi_{n,m}(p(\vec{a}))\}$ (as defined in the previous subsection). It suffices to show that this transition matrix is upper triangular with $\pm 1$ on the diagonal, with respect to some ordering on the rows and columns. More formally, we wish to show: 

\begin{proposition}\label{prop:exists}
There exists a partial order $\prec_{n,m}$ on the set of nonincreasing $K_{n,m}$-parking functions such that for all $\vec{a}$:
\[\phi_{n,m}(p(\vec{a})) = \pm x^{\vec{a}} + \sum_{\vec{b}} c_{\vec{a}, \vec{b}}x^{\vec{b}}\]
for some constants $c_{\vec{a}, \vec{b}}$, where the sum is over $K_{n,m}$-parking functions $\vec{b}$ such that $\mathrm{sort}(\vec{b}) \prec_{n,m}\vec{a}$. 
\end{proposition}

The goal of this subsection is to define the ordering $\prec_{n,m}$. To do so, we first introduce the \textit{row-column sequence} of a weakly nonintersecting pair:

\begin{definition}\label{def:row_column}
    The row-column sequence of a weakly nonintersecting pair $(P, Q)$ is computed as follows. Read rows of the grid from bottom to top. For each row $i$, record:
    \begin{enumerate}
        \item The number $\lambda_i$ of grid squares to the right of $Q$ in row $i$
        \item The number $\mu_i$ of grid squares to the left of $P$ in row $i$ and
        \item The number $c_i$ of grid squares above $P$ in column $m-\mu_i$ (the `length of the column above $P$'). 
    \end{enumerate}
    Then, the row-column sequence of $(P, Q)$ is the tuple 
    \[(\lambda_1, \mu_1, c_1, \lambda_2, \mu_2, c_2, \dots, \lambda_n, \mu_n, c_n)\]
\end{definition}

\begin{example}\label{ex:row-column}
Let's compute the row-column sequence of the weakly nonintersecting pair in Figure \ref{fig:420}. Starting at the bottom-most row, we see two squares right of the lower path, so $\lambda_1 = 2$. We see one labeled square to the left of the upper path, so $\mu_1 = 1$. Then, in the `column' right of the square labeled $(10, 5)$, we see $3$ squares above $P$, so $c_1 = 3$. 

Continuing like this for each row and then concatenating the results, we obtain 

\[(2, 1, 3, 2, 1, 3, 1, 2, 2, 1,3,0,0,3,0)\]

\end{example}

Now, we can define $\prec_{n,m}$:

\begin{definition}\label{def:partialorder}
    Let $(P, Q)$ and $(P', Q')$ be weakly nonintersecting pairs in an $(n-1) \times m$ grid. We say that 
    \[(P, Q) \prec_{n,m} (P', Q')\]
    if the row-column sequence of $(P', Q')$ \textit{lexicographically precedes} that of $(P, Q)$. 

    For nonincreasing $K_{n,m}$-parking functions $\vec{a}$ and $\vec{b}$, we say that $\vec{a} \prec_{n,m} \vec{b}$ if $(P_{\vec{a}}, Q_{\vec{a}}) \prec_{n,m} (P_{\vec{b}}, Q_{\vec{b}})$. 
\end{definition}

Now that we have our candidate $\prec_{n,m}$, the last step is to use it to prove Proposition \ref{prop:exists}. 

\subsection{Proof of Proposition \ref{prop:exists}}

To illustrate the main idea, we explain an extended example before giving the formal proof. 

\begin{example}\label{example:big_example}
    Let $n = 6, m = 5$, and let $\vec{a} = (4,2,0,0,0,3,3,2,1,1)$ be a nonincreasing $K_{6,5}$-parking function. First, we compute the corresponding weakly nonintersecting pair $(P_{\vec{a}}, Q_{\vec{a}})$ in the $5 \times 5$ grid; it is shown in Figure \ref{fig:420} (along with its grid square labeling). We computed its row-column sequence in Example \ref{ex:row-column}:

    \[(2, 1, 3, 2, 1, 3, 1, 2, 2, 1,3,0,0,3,0)\]

\begin{figure}
    \begin{center}
        \begin{tikzpicture}[scale=1.2]
    \foreach \x in {0,...,5}
        \draw (\x,0) -- (\x,5);

    \foreach \y in {0,...,5}
        \draw (0,\y) -- (5,\y);

    \draw[line width = 3] (0,0) --++ ({0}:3) --++ ({90}:2) --++ ({0}:1) --++ ({90}:2) --++ ({0}:1) --++ ({90}:1);

    \draw[line width = 3]
    (0,0) --++({0}:1) --++ ({90}:2) --++ ({0}:1) --++ ({90}:1) --++ ({0}:1) --++ ({90}:2) --++ ({0}:2);
    \draw (4.5, 0.5) node{$(1,11)$};
    \draw (4.5, 1.5) node{$(1,10)$};
    \draw (4.5, 2.5) node{$(1,9)$};
    \draw (4.5, 3.5) node{$(1,8)$};

    \draw (3.5, 0.5) node{$(2,11)$};
    \draw (3.5, 1.5) node{$(2,10)$};

    \draw (0.5, 0.5) node{$(10,5)$};
    \draw (0.5, 1.5) node{$(9,5)$};
    \draw (0.5, 2.5) node{$(8,5)$};
    \draw (0.5, 3.5) node{$(7,5)$};
    \draw (0.5, 4.5) node{$(6,5)$};

    \draw (1.5, 2.5) node{$(8,4)$};
    \draw (1.5, 3.5) node{$(7,4)$};
    \draw (1.5, 4.5) node{$(6,4)$};

    \draw (2.5, 3.5) node{$(7,3)$};
    \draw (2.5, 4.5) node{$(6,3)$};

\end{tikzpicture}
\end{center}
\caption{The weakly nonintersecting pair $(P_{\vec{a}}, Q_{\vec{a}})$ when $\vec{a} = (4,2,0,0,0,3,3,2,1,1)$}\label{fig:420}
\end{figure}

    To $\vec{a}$, we associate the polynomial $p(\vec{a}) \in \mathrm{Slim}_{6,5}$. Explicitly, $p(\vec{a})$ is
    
    \[(x_1-x_8)(x_1-x_9)(x_1 - x_{10})(x_1-x_{11})(x_2-x_{10})\dots (x_5-x_9)(x_5-x_{10})\]
    
    Applying $\phi_{6,5}$ to $p(\vec{a})$ returns a polynomial in $\mathrm{Park}_{6,5}$. When we expand $\phi_{6,5}(p(\vec{a}))$ in the monomial basis of $\mathrm{Park}_{6,5}$, we wish to show that $x^{\vec{a}}$ is the unique leading term with respect to $\prec_{6,5}$ (with coefficient $\pm 1$). 

    To do so, we should understand which monomials appear in $\phi_{6,5}(p(\vec{a}))$. Choosing a subset $S$ of the labeled grid squares tells us how to choose a term from each linear factor of $p(\vec{a})$. Specifically, given a subset $S$ of the labeled grid squares, we construct a monomial $x^{\vec{b}}$ as follows:
    \begin{enumerate}
    \item Start with $x^{\vec{b}} = 1$.
    \item For each grid square \textit{in} $S$ with label $(i,j)$, multiply your monomial by $x_j$. 
    \item For each grid square \textit{not} in $S$ with label $(i,j)$, multiply your monomial by $x_i$. 
        \item Test whether the resulting exponent vector $\vec{b}$ is a valid $K_{n,m}$-parking function. (To do so, compute $\mathrm{sort}(\vec{b})$ and test whether $(P_{\vec{b}}, Q_{\vec{b}})$ is weakly nonintersecting). If not, discard it (the monomial gets quotiented out by $\phi_{6,5}$). Otherwise, it contributes a term to $\phi_{n,m}(p(\vec{a}))$.  
    \end{enumerate}

    Every monomial of $\phi_{6,5}(p(\vec{a}))$ arises from a choice of $S$ in this way. Notice that when $S = \emptyset$, our is exactly $x^{\vec{a}}$. 

    Fix a choice of $S$ and let $x^{\vec{b}}$ be the resulting monomial. By reading grid squares from bottom to top and right to left, we will inductively argue that if any square is in $S$, either $\vec{b} \prec_{6,5}\vec{a}$ or $\phi_{6,5}(x^{\vec{b}}) = 0$. Let $(\lambda_1', \mu_1', c_1', \dots)$ be the row-column sequence of $\vec{b}$, and $(\lambda_1, \mu_1, \dots)$ be the row-column sequence of $\vec{a}$. 

    Start with the bottom right squares labeled $(2, 11)$ and $(1, 11)$ (corresponding to factors $(x_2-x_{11})$ and $(x_1-x_{11})$). If we choose $x_{11}$ from either factor, the entire monomial is sent to $0$ by $\phi_{6,5}$, so assume that neither square is included in $S$. 

    Now, consider the square labeled $(10, 5)$. Notice that there is no square labeled $(5, i)$ underneath $Q_{\vec{a}}$ for any $i$. Therefore, if we include this square in $S$, there would be at least three nonzero coordinates in $\mathrm{sort}(\vec{b})$ (by the previous paragraph, there are at least two). Therefore, we would have $\lambda_1' \geq 3 > 2 = \lambda_1$, and the row-column sequence of $\vec{a}$ would lexicographically precede that of $\mathrm{sort}(\vec{b})$.

    In this specific example, the same logic shows that \textit{none} of the squares above $P_{\vec{a}}$ can be included in $S$. (There is no label of the form $(5, i), (4,i)$ or $(3,i)$ under $Q_{\vec{a}}$, so $\lambda_1'$ would be at least $3$). 
    
    The remaining squares to consider are those labeled $(2,10), (1,10), (1,9)$ and $(1,8)$. First, assume that a square labeled $(i, 10)$ is included in $S$. There is only one column above $P_{\vec{a}}$ with a label of the form $(10, j)$. So, in $\mathrm{sort}(\vec{b})$, we would have $c_1' = 4 > 3 = c_1$. Since $\lambda_1' = \lambda_1$ and $\mu_1' = \mu_1$, the row-column sequence of $\vec{a}$ would again lexicographically precede that of $\mathrm{sort}(\vec{b})$. Similar logic holds for the square labeled $(1, 9)$: there is only one column above $P_{\vec{a}}$ with a square labeled $(9, i)$ for some $i$. 

    From the previous paragraph, we deduce that $\lambda_2 = \lambda_2', \mu_2 = \mu_2'$, $c_2 = c_2'$, and $\lambda_3' = \lambda_3$. It remains to consider the square labeled $(1,8)$. If it is included in $S$, then, because there are only two columns above $P_{\vec{a}}$ with labels of the form $(8, i)$, we would have $\mu_3' = 3$. But then, $\mu_3' > \mu_3 = 2$: because this is the first point in which the row-column sequences of $\vec{a}$ and $\mathrm{sort}(\vec{b})$ would differ, we would have $\mathrm{sort}(\vec{b}) \prec_{n,m}\vec{a}$. This square is the last labeled grid square, so we have shown $S = \emptyset$. 

    By considering each labeled grid square in a certain order, we have concluded that no grid square can be included in $S$ without forcing $\mathrm{sort}(\vec{b}) \prec_{n,m}\vec{a}$ or $\phi_{6,5}(x^{\vec{b}}) = 0$. Therefore, $x^{\vec{a}}$ is indeed the leading term of $\phi_{6,5}(p(\vec{a}))$ with respect to $\prec_{6,5}$ in this example. 

\end{example}

Having explained this extended example, we proceed to the proof of Proposition \ref{prop:exists}. The following simple lemma explaining how $\mathrm{sort}(\vec{b})$ compares to $\vec{a}$ will be useful.

\begin{lemma}\label{lemma:tech}
    Let $\vec{a}$ be a nonincreasing $K_{n,m}$-parking function. If $(i,j)$ is the label of some grid square, consider $\vec{b} = \vec{a} + (e_j-e_i)$ (where $e_i$ is the $i$th coordinate vector). Then, $(P_{\mathrm{sort}(\vec{b})}, Q_{\mathrm{sort}(\vec{b})})$ can be computed as follows. 
    \begin{enumerate}
        \item If the square labeled $(i,j)$ was above $P_{\vec{a}}$: remove that square, re-sorting the rows above $P_{\vec{a}}$ in increasing order if necessary. Add a new square below $Q_{\vec{a}}$. This square is added to the \textit{lowest} row $k$ such that $(j, k)$ is not a label underneath $Q_{\vec{a}}$. 
        \item If the square labeled $(i,j) $ was below $Q_{\vec{a}}$: remove that labeled square (re-sorting columns in increasing order if necessary). Add a new grid square above $P_{\vec{a}}$. This square is added to the \textit{leftmost} column $k$ such that $(j, k)$ is not a label underneath $Q_{\vec{a}}$. 
    \end{enumerate}
\end{lemma}

Lemma \ref{lemma:tech} follows from unwinding Definition \ref{def:labeling} and the bijection $\vec{a} \leftrightarrow (P_{\vec{a}}, Q_{\vec{a}})$. Finally, we prove Proposition \ref{prop:exists}. 

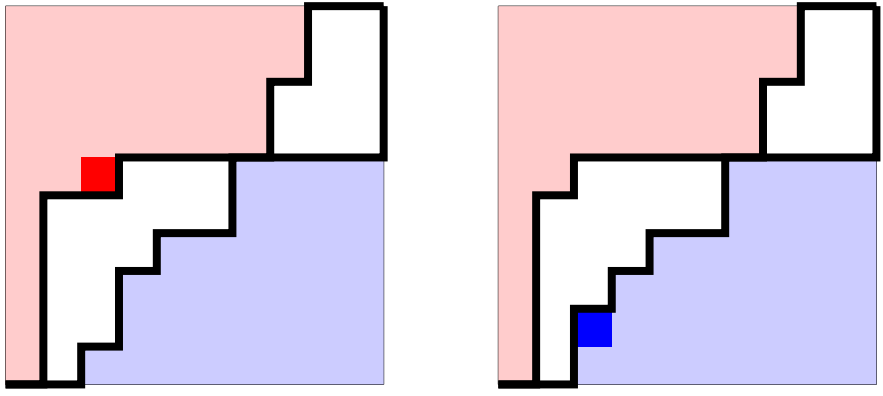
\begin{figure}
    \centering
\begin{tikzpicture}
    \draw (0,0) -- (0,5);
    \draw (0,0) -- (5,0);
    \draw (5,0) -- (5,5);
    \draw (0,5) -- (5,5);

    \fill[red!20](0, 0) -- (0.5, 0) -- (0.5, 2.5) -- (1.5, 2.5) -- (1.5, 3) --(3.5, 3) -- (3.5, 4) -- (4, 4) -- (4, 5) -- (5, 5) -- (0, 5) --cycle;

    \fill[blue!20](0, 0) -- (1, 0) -- (1, 0.5) -- (1.5, 0.5) -- (1.5, 1) -- (1.5, 1.5) -- (2, 1.5) -- (2, 2) -- (3,2) -- (3, 3) -- (5, 3) -- (5, 5) -- (5, 0) -- cycle;

    \fill[red!] (1, 2.5) -- (1.5, 2.5) -- (1.5, 3) -- (1, 3) -- cycle;

    \draw[line width=3] (0, 0) -- (0.5, 0) -- (0.5, 2.5) -- (1.5, 2.5) -- (1.5, 3) --(3.5, 3) -- (3.5, 4) -- (4, 4) -- (4, 5) -- (5, 5);

    \draw[line width=3]
    (0, 0) -- (1, 0) -- (1, 0.5) -- (1.5, 0.5) -- (1.5, 1) -- (1.5, 1.5) -- (2, 1.5) -- (2, 2) -- (3,2) -- (3, 3) -- (5, 3) -- (5, 5);

\end{tikzpicture} \hspace{3em} \begin{tikzpicture}
    \draw (0,0) -- (0,5);
    \draw (0,0) -- (5,0);
    \draw (5,0) -- (5,5);
    \draw (0,5) -- (5,5);

    \fill[red!20](0, 0) -- (0.5, 0) -- (0.5, 2.5) -- (1, 2.5) -- (1, 3) -- (1.5, 3) --(3.5, 3) -- (3.5, 4) -- (4, 4) -- (4, 5) -- (5, 5) -- (0, 5) --cycle;

    \fill[blue!20](0, 0) -- (1, 0) -- (1, 0.5) -- (1.5, 0.5) -- (1.5, 1) -- (1.5, 1.5) -- (2, 1.5) -- (2, 2) -- (3,2) -- (3, 3) -- (5, 3) -- (5, 5) -- (5, 0) -- cycle;

    \fill[blue!](1,0.5) --(1.5, 0.5) -- (1.5, 1) -- (1, 1) -- cycle;

    \draw[line width=3] (0, 0) -- (0.5, 0) -- (0.5, 2.5) -- (1, 2.5) -- (1, 3) -- (1.5, 3) --(3.5, 3) -- (3.5, 4) -- (4, 4) -- (4, 5) -- (5, 5);

    \draw[line width=3]
    (0, 0) -- (1, 0) -- (1, 1) -- (1.5, 1) -- (1.5, 1.5) -- (2, 1.5) -- (2, 2) -- (3,2) -- (3, 3) -- (5, 3) -- (5, 5);

\end{tikzpicture}

    \caption{An example of Lemma \ref{lemma:tech}. On the left is $(P_{\vec{a}}, Q_{\vec{a}})$ for some $\vec{a}$. We highlight the square with label $(i,j)$. On the right is $(P_{\vec{b}}, Q_{\mathrm{sort}(\vec{b})})$, where $\vec{b} = \vec{a} + (e_j - e_i)$. The highlighted square finds the bottom-most row such that $(j, k)$ is not a label below $Q_{\vec{a}}$ and adds one to the length of that row.}
\end{figure}

\begin{proof}
    As we explained in Example \ref{example:big_example}, every monomial in the basis expansion of $\phi_{n,m}(p(\vec{a}))$ arises from a choice of subset $S$ of the labeled grid squares of $(P_{\vec{a}}, Q_{\vec{a}})$. The label $(i, j)$ corresponds to the linear factor $ \pm (x_i - x_j)$ of $p(\vec{a})$, and whether a grid square is included in $S$ or not tells us whether to choose the term $x_j$ or $x_i$. It is not hard to verify that $S = \emptyset$ is the only possible choice yielding $x^{\vec{a}}$ itself (with coefficient $\pm 1$). It remains to show that for any $S \neq \emptyset$, the resulting term $x^{\vec{b}}$ is either sent to $0$ under $\phi_{n,m}$ or satisfies $\mathrm{sort}(\vec{b}) \prec_{n,m}\vec{a}$, where $\prec$ is as in Definition \ref{def:partialorder}. 

    Consider a $\vec{b}$ arising from an $S$ in this way such that $\mathrm{sort}(\vec{b})$ is not (strictly) less than $\vec{a}$ and $\phi_{n,m}(x^{\vec{b}}) \neq 0$. We wish to show that $\vec{b} = \vec{a}$. As in Example \ref{example:big_example}, we will inductively consider the labeled grid squares from bottom to top and right to left, concluding that none of them can be included in $S$. 

    Let $(\lambda_1, \mu_1, c_1, \dots)$ be the row-column sequence of $\vec{a}$, and $(\lambda_1', \mu_1', c_1', \dots)$ be the row-column sequence of $\mathrm{sort}(\vec{b})$. 

    \textbf{Base case:} Consider the $\lambda_1$ grid squares in row $1$ to the right of $Q_{\vec{a}}$; they are labeled $(i, n+m)$ for varying $i$. If any of these squares is included in $S$, then our monomial $x^{\vec{b}}$ is multiplied by $x_{n+m}$. But $\phi_{n,m}$ specializes $x_{n+m}$ to $0$, and would thus send $x^{\vec{b}}$ to $0$. Therefore, none of these squares can be included in $S$. 

Note that this base case implies $\lambda_1' \geq \lambda_1'$ (and since $\mathrm{sort}(\vec{b})$ is not less than $\vec{a}$, we must therefore have $\lambda_1 = \lambda_1'$). 

Now, fix a $j \leq m$. We inductively assume that all grid squares right of $Q_{\vec{a}}$ in row at most $j$ are not in $S$. We will next show that all grid squares right of $Q_{\vec{a}}$ in row $j+1$ also cannot be in $S$. Inductively, we will conclude that none of the squares right of $Q_{\vec{a}}$ can be in $S$. 

\begin{claim}\label{base_case_claim}
    Suppose the leftmost square that is right of $Q_{\vec{a}}$ in row $j$ lies in column $k$. Then, any square left of $P_{\vec{a}}$ in column $> k$ \textit{cannot} be in $S$ either. 
\end{claim}

To prove the claim, assume otherwise. Then, there is some square labeled $(i, k')$ in $S$ such that $k' > k$. There is no label of the form $(k', n+m+1-j)$ in row $j$. Apply Lemma \ref{lemma:tech} and the fact that all squares in lower rows are `fixed' to deduce that, in $(P_{\mathrm{sort}(\vec{b})}, Q_{\mathrm{sort}(\vec{b})})$, we have $\lambda_j' > \lambda_j$ (or $\lambda_l' > \lambda_l$ for some $l < j$). By the definition of $\prec$, this means $\mathrm{sort}(\vec{b}) \prec_{n,m}\vec{a}$, a contradiction. 

\begin{figure}
    \centering
\begin{tikzpicture}

    \draw (0,0) -- (0,5);
    \draw (0,0) -- (5,0);
    \draw (5,0) -- (5,5);
    \draw (0,5) -- (5,5);

    \fill[gray!20] (0,0) -- (1.5, 0) -- (1.5, 1.5) -- (2.5, 1.5) -- (2.5, 2.5) -- (5, 2.5) -- (5, 0) -- cycle;

    \fill[gray!20] (0, 0) -- (0, 2) -- (0.5, 2) -- (0.5, 2.5) -- (2, 2.5) -- (2, 4.5) -- (2.5, 4.5) -- (2.5, 5) -- (0, 5) -- cycle;

    \draw[line width =1] (2.5, 2.5) -- (5, 2.5);

    \draw[line width =1] (2.5, 4.5) -- (2.5, 5);
    
    \draw[line width=3] (0,0) -- (1.5, 0) -- (1.5, 1.5) -- (2.5, 1.5) -- (2.5, 3.5) -- (3, 3.5) -- (3, 4) -- (4, 4) -- (4, 5) -- (5,5);

    \draw[line width=3] (0, 0) -- (0, 2) -- (0.5, 2) -- (0.5, 2.5) -- (2, 2.5) -- (2, 4.5) -- (3, 4.5) -- (3, 5) -- (5, 5);
    
\end{tikzpicture}
    \caption{An example of Claim \ref{base_case_claim}: suppose all the shaded squares below the lower path are \textit{not} in $S$ (they are `fixed.') Then, all the shaded squares above the upper path must also not be in $S$ - otherwise, according to Lemma \ref{lemma:tech}, they would `fall down' and lengthen one of the shaded rows.}
\end{figure}

Given Claim \ref{base_case_claim}, there are now two cases to consider. 

\textbf{Case 1: } the row length $\mu_{j}$ is \textit{strictly} less than the row length $\mu_{j+1}$.  

\begin{figure}
    \centering
\begin{tikzpicture}

    \fill[red!40] (4, 1) -- (4.5, 1) -- (4.5, 1.5) -- (4, 1.5) -- cycle;

    \fill[gray!30] (3, 0) -- (3, 1) -- (6, 1) -- (6, 0) -- cycle;

    \draw[line width = 3] (0, 0) -- (0, 1) -- (2, 1) -- (2, 2) -- (4, 2) -- (4, 3);
    \draw[line width = 3] (1, 0) -- (3, 0) -- (3, 1) -- (4, 1) -- (4, 1.5) -- (6, 1.5);

    \draw[line width = 3, ->] (5.5, 0.5) -- (7, 0.5);
\fill[gray!30] (10.5, 0) -- (10.5, 1) -- (13.5, 1) -- (13.5, 0) -- cycle;

    \fill[red!40] (7.5, 0.5) -- (8, 0.5) -- (8, 1) -- (7.5, 1);

    \draw[line width = 3] (7.5, 0) -- (7.5, 0.5) -- (8, 0.5) -- (8, 1) -- (9.5, 1) -- (9.5, 2) -- (11.5, 2) -- (11.5, 3);
    \draw[line width = 3] (8.5, 0) -- (10.5, 0) -- (10.5, 1) -- (12, 1) -- (12, 1.5) -- (13.5, 1.5);

\end{tikzpicture}
    \caption{In case $1$, we already know that all squares below $Q_{\vec{a}}$ in row $\leq j$ are \textit{not} in $S$ (the lightly shaded region). Suppose a square in row $j+1$ (highlighted) is included into $S$. Then, it `jumps down' to row $j$ as shown, increasing the length of $\mu_j$.}\label{fig:case1}
\end{figure}

In this case, suppose that a square in row $j+1$ right of $Q_{\vec{a}}$ is included into $S$. Then, in $P_{\mathrm{sort}(\vec{b})}$, there would be at least one additional square in column $m-\mu_j$ (since that column is the leftmost column without a label of the form $(j+1, k)$ for some $k$). In other words, we would have $\mu_j' > \mu_j$ in the resulting $P_{\mathrm{sort}(\vec{b})}$, as shown in Figure \ref{fig:case1}. Therefore, the row-column sequence of $\vec{a}$ would lexicographically precede that of $\vec{b}$, a contradiction. 

\textbf{Case 2:} the row length $\mu_{j}$ is \textit{equal} to the row length $\mu_{j+1}$.

\begin{figure}
    \centering

    \begin{tikzpicture}

    \fill[red!40] (4, 1) -- (4.5, 1) -- (4.5, 1.5) -- (4, 1.5) -- cycle;

    \fill[gray!30] (3, 0) -- (3, 1) -- (6, 1) -- (6, 0) -- cycle;
    
        \draw[line width = 3] (1, 0) -- (3, 0) -- (3, 1) -- (4, 1) -- (4, 1.5) -- (6, 1.5) -- (6,3);

        \draw[line width = 3] (0, 0) -- (0, 2) -- (2,2) -- (2, 3) -- (6, 3);

    \draw[line width = 3, ->](6.5, 1.5) -- (7.5, 1.5); 

\fill[gray!30] (11, 0) -- (11, 1) -- (14, 1) -- (14, 0) -- cycle;

\fill[red!40] (8,1.5) -- (8, 2) -- (8.5, 2) -- (8.5, 1.5);
    
        \draw[line width = 3] (9, 0) -- (11, 0) -- (11, 1) -- (12.5, 1) -- (12.5, 1.5) -- (14, 1.5) -- (14,3);

        \draw[line width = 3] (8, 0) -- (8, 1.5) -- (8.5, 1.5) -- (8.5, 2) -- (10,2) -- (10, 3) -- (14, 3);
    
\end{tikzpicture}

    \caption{In case $2$: suppose again that a square in row $j+1$ is included into $S$. This time, it `jumps up' to the top of the empty column next to $P_{\vec{a}}$, lengthening an earlier column $c_i$ for $i < j+1$.}\label{fig:case2}
\end{figure}
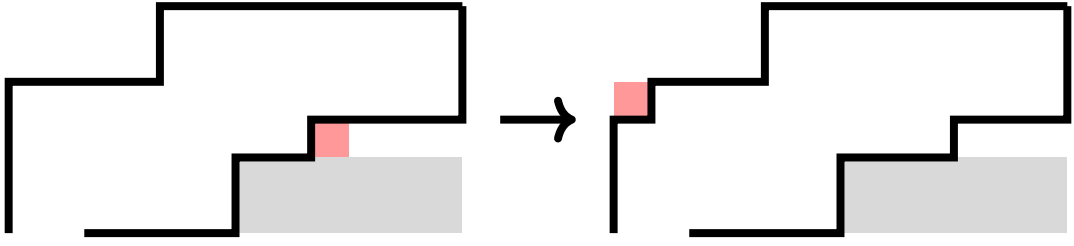

Again, suppose that a square in row $j+1$ right of $Q_{\vec{a}}$ is included into $S$. In the resulting $P_{\mathrm{sort}(\vec{a})}$, a new labeled grid square would appear at the top of column $m-\mu_{j}$, as shown in Figure \ref{fig:case2}. This transformation decreases the value of $c_{j}$ by at least one, again lexicographically increasing the row-column sequence. 

These two cases are exhaustive, so we have inductively shown that no grid square below $Q_{\vec{a}}$ can be included into $S$. Similar logic shows that no grid square above $P_{\vec{a}}$ can be included into $S$. Therefore, we must have $\vec{b} = \vec{a}$, as desired. 

\end{proof}

The proof of Proposition \ref{prop:exists} completes the proof of Theorem \ref{thm:main}.

\section{Open problems}

We end by proposing a more general conjecture:

\begin{conjecture}\label{conj}
    For any graph simple graph $G$ with a distinguished vertex $v$, we have 
    \[\mathrm{Res}^{\mathrm{aut}(G)}_{\mathrm{aut}_v(G)} (\mathrm{Slim}_G) \simeq \mathrm{Park}_G\]
\end{conjecture}

Conjecture \ref{conj} is true when $G = K_{n+1}$ and $K_{n,m}$ by Theorems \ref{thm:berget-rhoades} and \ref{thm:main}. In other special cases, it can be quickly verified; for instance, it is trivial when $G$ is a tree and easy to check when $G$ is a single cycle. We have also tested Conjecture \ref{conj} with code for many small graphs. 

In general, we can always define the homomorphism $\phi_G$ that we describe in Subsection \ref{subsec:phi}. What needs to be shown is that $\phi_G$ is bijective. There are two main steps that need more ideas in general: 

\begin{enumerate}
    \item How do we choose natural orbit representatives for arbitrary graphs? 
    \item Which partial order generalizes $\prec_{n,m}$? (For $K_n$, Berget-Rhoades use graded rev-lex order on partitions. They note that, perhaps surprisingly, the dominance order on partitions does \textit{not} work for their argument). 
\end{enumerate}

A natural next class of graphs to consider is the \textit{complete multipartite graph} $K_{n_1, n_2 \dots n_m}$ with distinct part sizes. For these graphs, the first point becomes easy: again take `nonincreasing' functions, as $\mathrm{aut}(G)$ is still a direct product of symmetric groups. However, we were not able to come up with the right generalization of $\prec_{n,m}$. 

We also wonder whether there is an analogue of Konvalinka-Tewari's constructions \cite{tewari}, as well as their conjecture about $\mathrm{Slim}_n$, in the bipartite setting. 

\bibliographystyle{alpha}
\bibliography{refs}

\end{document}